\documentclass[11pt]{amsart}

\usepackage[margin=1.25in]{geometry}
\usepackage{graphicx}
\usepackage{amsmath, amssymb, amsthm, graphicx, enumerate, tikz, float, color, bbm}
\usepackage{mathtools, comment}
\usepackage{wasysym}
\usepackage[misc]{ifsym}
\usepackage[colorlinks]{hyperref}

\hypersetup{citecolor=blue}
\usetikzlibrary{matrix,arrows,decorations.pathmorphing}

\newtheorem{theorem}{Theorem}[section]
\newtheorem{lemma}[theorem]{Lemma}
\newtheorem{proposition}[theorem]{Proposition}
\newtheorem{corollary}[theorem]{Corollary}

\theoremstyle{definition}
\newtheorem{definition}[theorem]{Definition}
\newtheorem{example}[theorem]{Example}

\theoremstyle{remark}
\newtheorem{remark}[theorem]{Remark}

\numberwithin{equation}{section}

\DeclareMathOperator{\Ker}{Ker}

\DeclareMathOperator{\Ig}{Im}

\makeatletter
\@namedef{subjclassname@2020}{%
  \textup{2020} Mathematics Subject Classification}
\makeatother

\begin{document}

\allowdisplaybreaks

\title{Secondary Hochschild homology and differentials}

\author{Jacob Laubacher}
\address{Department of Mathematics, St. Norbert College, De Pere, WI 54115}
\email{jacob.laubacher@snc.edu}


\date{\today}

\keywords{Hochschild homology, cyclic homology, differentials.\\\indent\emph{Corresponding author.} Jacob Laubacher \Letter~\href{mailto:jacob.laubacher@snc.edu}{jacob.laubacher@snc.edu} \phone~920-403-2961.}

\begin{abstract}
In this paper we study a generalization of K\"ahler differentials, which correspond to the secondary Hochschild homology associated to a triple $(A,B,\varepsilon)$. We establish computations in low dimension, while also showing how this connects with the kernel of a multiplication map.
\end{abstract}

\maketitle

%%%%%%%%%%%%%%%%%%%%%%%%%%%%%%%%%%%%%%%%%%%%%%%%%%%%%%%%%%%%%%%%%%%%%%%%%%%%
\section{Introduction}
%%%%%%%%%%%%%%%%%%%%%%%%%%%%%%%%%%%%%%%%%%%%%%%%%%%%%%%%%%%%%%%%%%%%%%%%%%%%

Gerhard Hochschild introduced the cohomology theory bearing his name in \cite{H} in order to study extensions of algebras over a field $\mathbbm{k}$. Later, Gerstenhaber used this theory to study deformations in \cite{G}. On the other hand, Hochschild homology led to the study of cyclic homology, as well as a stepping-stone towards de Rham cohomology.

Secondary Hochschild (co)homology, a generalization of Hochschild (co)homology, has now been studied for some time. The cohomology originally introduced in \cite{S} by Staic, and the homology in \cite{LSS}, the theory involves a triple $(A,B,\varepsilon)$ which consists of a commutative $\mathbbm{k}$-algebra $B$ inducing a $B$-algebra structure on an associative $\mathbbm{k}$-algebra $A$ by way of a morphism $\varepsilon:B\longrightarrow A$. The corresponding secondary cyclic (co)homology was also studied in \cite{LSS}, and properties, applications, and generalizations of the secondary Hochschild (co)homology were investigated in \cite{BBN}, \cite{CHL}, \cite{CSS}, \cite{DMN}, \cite{Laub}, \cite{L}, and \cite{SS}, among others.

In this paper, our goal is to study how the secondary Hochschild homology associated to the triple $(A,B,\varepsilon)$ interacts with a generalization of K\"ahler differentials. In particular, we prove Theorem \ref{newmain}, which is a generalized version of Theorem \ref{oldmain}. In Section \ref{sec2} we provide all necessary background information, keeping this paper as self-contained as possible. Section \ref{sec3} introduces our notion of secondary K\"ahler differentials, leading to an isomorphism between them and the secondary Hochschild homology under certain conditions, as well as providing computations in low dimension. Next, Section \ref{sec4} discusses the kernel of the multiplication map from $A^{\otimes2}\otimes B$ to $A$ and how this also corresponds to the aforementioned differentials. We round out the paper in Section \ref{sec5} by hypothesizing possible avenues for future work.

%%%%%%%%%%%%%%%%%%%%%%%%%%%%%%%%%%%%%%%%%%%%%%%%%%%%%%%%%%%%%%%%%%%%%%%%%%%%
\section{Preliminaries}\label{sec2}
%%%%%%%%%%%%%%%%%%%%%%%%%%%%%%%%%%%%%%%%%%%%%%%%%%%%%%%%%%%%%%%%%%%%%%%%%%%%

We start by fixing $\mathbbm{k}$ to be a field containing $\mathbb{Q}$. We set all tensor products over $\mathbbm{k}$, unless otherwise stated (so $\otimes=\otimes_\mathbbm{k})$. Next, let $A$ be an associative $\mathbbm{k}$-algebra with multiplicative unit. Most of our work is done over a triple, the formal definition of which is below:

\begin{definition}(\cite{S})
We call $(A,B,\varepsilon)$ a \textbf{triple} if $A$ is a $\mathbbm{k}$-algebra, $B$ is a commutative $\mathbbm{k}$-algebra, and $\varepsilon:B\longrightarrow A$ is a morphism of $\mathbbm{k}$-algebras such that $\varepsilon(B)\subseteq\mathcal{Z}(A)$. Call $(A,B,\varepsilon)$ a \textbf{commutative triple} if $A$ is also commutative.
\end{definition}

One can find examples of triples in \cite{S} and \cite{SS}, as well as more extensions in \cite{CL} and \cite{CLS}.

%%%%%%%%%%%%%%%%%%%%%%%%%%%%%%%%%%%%%%%%%%%%%%%%%%%%%%%%%%%%%%%%%%%%%%%%%%%%
\subsection{The secondary Hochschild homology}\label{firstsub}
%%%%%%%%%%%%%%%%%%%%%%%%%%%%%%%%%%%%%%%%%%%%%%%%%%%%%%%%%%%%%%%%%%%%%%%%%%%%

Next we recall the secondary Hochschild homology associated to the triple $(A,B,\varepsilon)$, which was introduced in \cite{LSS}.

Following notation in \cite{LSS}, we define $\overline{C}_n(A,B,\varepsilon)=A^{\otimes n+1}\otimes B^{\otimes\frac{n(n+1)}{2}}$ for all $n\geq0$. As is customary in these settings (see \cite{CS} for example), we can follow an organizational schematic to represent the elements of $A^{\otimes n+1}\otimes B^{\otimes\frac{n(n+1)}{2}}$ as
$$
\otimes
\begin{pmatrix}
a_0 & b_{0,1} & b_{0,2} & \cdots & b_{0,n-2} & b_{0,n-1} & b_{0,n}\\
1 & a_1 & b_{1,2} & \cdots & b_{1,n-2} & b_{1,n-1} & b_{1,n}\\
1 & 1 & a_2 & \cdots & b_{2,n-2} & b_{2,n-1} & b_{2,n}\\
\vdots & \vdots & \vdots & \ddots & \vdots & \vdots & \vdots\\
1 & 1 & 1 & \cdots & a_{n-2} & b_{n-2,n-1} & b_{n-2,n}\\
1 & 1 & 1 & \cdots & 1 & a_{n-1} & b_{n-1,n}\\
1 & 1 & 1 & \cdots & 1 & 1 & a_n\\
\end{pmatrix},
$$
where $a_i\in A$, $b_{i,j}\in B$, and $1\in \mathbbm{k}$.

Next, we have that $\partial_n:\overline{C}_n(A,B,\varepsilon)\longrightarrow \overline{C}_{n-1}(A,B,\varepsilon)$ is given by
\begin{align*}
\partial_n&\Big(\otimes
\begin{pmatrix}
a_0 & b_{0,1} & \cdots & b_{0,n-1} & b_{0,n}\\
1 & a_1 & \cdots & b_{1,n-1} & b_{1,n}\\
\vdots & \vdots & \ddots & \vdots & \vdots\\
1 & 1 & \cdots & a_{n-1} & b_{n-1,n}\\
1 & 1 & \cdots & 1 & a_n\\
\end{pmatrix}
\Big)\\
&=\sum_{i=0}^{n-1}(-1)^i\left(\otimes
\begin{pmatrix}
a_0 & b_{0,1} & \cdots & b_{0,i}b_{0,i+1} & \cdots & b_{0,n-1} & b_{0,n}\\
1 & a_1 & \cdots & b_{1,i}b_{1,i+1} & \cdots & b_{1,n-1} & b_{1,n}\\
\vdots & \vdots & \ddots & \vdots & \ddots & \vdots & \vdots\\
1 & 1 & \cdots & a_i\varepsilon(b_{i,i+1})a_{i+1} & \cdots & b_{i,n-1}b_{i+1,n-1} & b_{i,n}b_{i+1,n}\\
\vdots & \vdots & \ddots & \vdots & \ddots & \vdots & \vdots\\
1 & 1 & \cdots & 1 & \cdots & a_{n-1} & b_{n-1,n}\\
1 & 1 & \cdots & 1 & \cdots & 1 & a_n\\
\end{pmatrix}\right)\\
&~+(-1)^n\left(\otimes
\begin{pmatrix}
a_n\varepsilon(b_{0,n})a_0 & b_{1,n}b_{0,1} & \cdots & b_{n-2,n}b_{0,n-2} & b_{n-1,n}b_{0,n-1}\\
1 & a_1 & \cdots & b_{1,n-2} & b_{1,n-1}\\
\vdots & \vdots & \ddots & \vdots & \vdots\\
1 & 1 & \cdots & a_{n-2} & b_{n-2,n-1}\\
1 & 1 & \cdots & 1 & a_{n-1}\\
\end{pmatrix}\right).
\end{align*}

It was again shown in \cite{LSS} that $\partial_{n-1}\circ\partial_n=0$, and that the resulting complex is denoted by $\overline{\mathbf{C}}_\bullet(A,B,\varepsilon)$.

\begin{definition}(\cite{LSS})
The homology of the chain complex $\overline{\mathbf{C}}_\bullet(A,B,\varepsilon)$ is called the \textbf{secondary Hochschild homology associated to the triple $(A,B,\varepsilon)$}, and this is denoted by $HH_*(A,B,\varepsilon)$.
\end{definition}

\begin{remark}(\cite{LSS})
By taking $B=\mathbbm{k}$, notice that the secondary Hochschild homology associated to the triple $(A,B,\varepsilon)$ reduces to the usual Hochschild homology associated to $A$. In notation, one has that $HH_n(A,\mathbbm{k},\varepsilon)=HH_n(A)$ for all $n\geq0$.
\end{remark}

Since it will be used later in the paper, it may be helpful to view the complex $\overline{\mathbf{C}}_\bullet(A,B,\varepsilon)$ in low dimension. In particular, we have
$$
\ldots\xrightarrow{~\partial_3~}A^{\otimes3}\otimes B^{\otimes3}\xrightarrow{~\partial_2~}A^{\otimes2}\otimes B\xrightarrow{~\partial_1~}A\longrightarrow0
$$
with
$$
\partial_1\Big(\otimes\begin{pmatrix} a & \alpha\\ 1 & b\\ \end{pmatrix}\Big)=a\varepsilon(\alpha)b-b\varepsilon(\alpha)a,
$$
and
$$
\partial_2\Big(\otimes\begin{pmatrix} a & \alpha & \beta\\ 1 & b & \gamma\\ 1 & 1 & c\\ \end{pmatrix}\Big)=\otimes\begin{pmatrix} a\varepsilon(\alpha)b & \beta\gamma\\ 1 & c\\ \end{pmatrix}-\otimes\begin{pmatrix} a & \alpha\beta\\ 1 & b\varepsilon(\gamma)c\\ \end{pmatrix}+\otimes\begin{pmatrix} c\varepsilon(\beta)a & \alpha\gamma\\ 1 & b\\ \end{pmatrix}.
$$

\begin{example}\label{dimzerH}(\cite{LSS})
For any triple $(A,B,\varepsilon)$, we have $HH_0(A,B,\varepsilon)=HH_0(A)=\frac{A}{[A,A]}$.
\end{example}

%%%%%%%%%%%%%%%%%%%%%%%%%%%%%%%%%%%%%%%%%%%%%%%%%%%%%%%%%%%%%%%%%%%%%%%%%%%%
\subsection{The secondary cyclic homology}
%%%%%%%%%%%%%%%%%%%%%%%%%%%%%%%%%%%%%%%%%%%%%%%%%%%%%%%%%%%%%%%%%%%%%%%%%%%%

Connes introduced cyclic cohomology in \cite{C}, extensive details of which can be found in \cite{Lo}. Here we recall the analogous secondary homological version that was established in \cite{LSS}. We start by considering the permutation $\lambda=(0, 1, 2,\ldots, n)$ and the cyclic group $C_{n+1}=\langle\lambda\rangle$. Notice that $C_{n+1}$ has a natural action on $\overline{C}_n(A,B,\varepsilon)$ given by
$$
\lambda\Big(\otimes
\begin{pmatrix}
a_0 & b_{0,1} & \cdots & b_{0,n-1} & b_{0,n}\\
1 & a_1 & \cdots & b_{1,n-1} & b_{1,n}\\
\vdots & \vdots & \ddots & \vdots & \vdots\\
1 & 1 & \cdots & a_{n-1} & b_{n-1,n}\\
1 & 1 & \cdots & 1 & a_n\\
\end{pmatrix}\Big)
=(-1)^n\left(\otimes
\begin{pmatrix}
a_n & b_{0,n} & \cdots & b_{n-2,n} & b_{n-1,n}\\
1 & a_0 & \cdots & b_{0,n-2} & b_{0,n-1}\\
\vdots & \vdots & \ddots & \vdots & \vdots\\
1 & 1 & \cdots & a_{n-2} & b_{n-2,n-1}\\
1 & 1 & \cdots & 1 & a_{n-1}\\
\end{pmatrix}\right).
$$

Consider the new complex built by setting $\overline{C}_n^{\lambda}(A,B,\varepsilon)=\overline{C}_n(A,B,\varepsilon)/\Ig(1-\lambda)$, while continuing to use the maps $\partial_n$ from Section \ref{firstsub}. It was verified in \cite{LSS} that $\partial_n$ maintains the necessary structure.

\begin{definition}(\cite{LSS})
The homology of the chain complex $\overline{\mathbf{C}}_\bullet^\lambda(A,B,\varepsilon)$ is called the \textbf{secondary cyclic homology associated to the triple $(A,B,\varepsilon)$}, and this is denoted by $HC_*(A,B,\varepsilon)$.
\end{definition}

\begin{remark}(\cite{LSS})
As before, taking $B=\mathbbm{k}$, we have that the secondary cyclic homology associated to the triple $(A,B,\varepsilon)$ reduces to the usual cyclic homology associated to $A$. That is, $HC_n(A,\mathbbm{k},\varepsilon)=HC_n(A)$ for all $n\geq0$.
\end{remark}

\begin{example}\label{dimzer}(\cite{LSS})
For any triple $(A,B,\varepsilon)$, we have that $HC_0(A,B,\varepsilon)=HH_0(A,B,\varepsilon)$.
\end{example}

Connes famously established a long exact sequence bearing his name that connected Hochschild homology with its cyclic counterpart. In \cite{LSS} the analogue of Connes' long exact sequence for the secondary case was produced:

\begin{theorem}\label{SCLES}\emph{(\cite{LSS})}
Let $\mathbbm{k}$ be a field of characteristic zero. For a triple $(A,B,\varepsilon)$, we have the long exact sequence
$$\ldots\xrightarrow{B_*}HH_n(A,B,\varepsilon)\xrightarrow{I_*}HC_n(A,B,\varepsilon)\xrightarrow{S_*}HC_{n-2}(A,B,\varepsilon)\xrightarrow{B_*}HH_{n-1}(A,B,\varepsilon)\xrightarrow{I_*}\ldots$$
\end{theorem}

%%%%%%%%%%%%%%%%%%%%%%%%%%%%%%%%%%%%%%%%%%%%%%%%%%%%%%%%%%%%%%%%%%%%%%%%%%%%
\subsection{The kernel of a multiplication map}\label{oldkernel}
%%%%%%%%%%%%%%%%%%%%%%%%%%%%%%%%%%%%%%%%%%%%%%%%%%%%%%%%%%%%%%%%%%%%%%%%%%%%

For a commutative algebra $A$, set $I$ to be the kernel of the multiplication map from $A\otimes A$ to $A$. That is, consider $m:A\otimes A\longrightarrow A$ given by $m(a\otimes b)=ab$, and then set $I=\Ker(m)$. One can show that $I$ is generated by elements of the form $1\otimes a-a\otimes1$, where $a\in A$, and that the $A$-bimodule $I/I^2$ is $A$-symmetric. Next, recalling K\"ahler differentials $\Omega_{A|\mathbbm{k}}^1$, we have the following classic result which unites our concepts of Hochschild homology, K\"ahler differentials, and the set $I$ (see \cite{W}, for instance):

\begin{theorem}\label{oldmain}
For a commutative algebra $A$, we have that
$$
HH_1(A)\cong\Omega_{A|\mathbbm{k}}^1\cong I/I^2.
$$
\end{theorem}

%%%%%%%%%%%%%%%%%%%%%%%%%%%%%%%%%%%%%%%%%%%%%%%%%%%%%%%%%%%%%%%%%%%%%%%%%%%%
\section{Secondary Differentials}\label{sec3}
%%%%%%%%%%%%%%%%%%%%%%%%%%%%%%%%%%%%%%%%%%%%%%%%%%%%%%%%%%%%%%%%%%%%%%%%%%%%

The main objective of this paper is to get a result like Theorem \ref{oldmain} that corresponds to the secondary case. The goal of this section is to establish the former isomorphism.

\begin{definition}\label{SKdiff}
For a commutative triple $\mathcal{T}=(A,B,\varepsilon)$, denote $\Omega_{\mathcal{T}|\mathbbm{k}}^1$ to be the left $B\otimes A$-module of \textbf{secondary K\"ahler differentials} generated by the $\mathbbm{k}$-linear symbols $d(\alpha\otimes a)$ for $\alpha\in B$ and $a\in A$ with the module structure of $(\alpha\otimes a)\cdot d(\beta\otimes b)=a\varepsilon(\alpha)d(\beta\otimes b)$, along with the relations
\begin{enumerate}[(i)]
\item\label{og1} $d(\lambda(\alpha\otimes a)+\mu(\beta\otimes b))=\lambda d(\alpha\otimes a)+\mu d(\beta\otimes b)$,
\item\label{og2} $d((\alpha\otimes a)(\beta\otimes b))=a\varepsilon(\alpha)d(\beta\otimes b)+b\varepsilon(\beta)d(\alpha\otimes a)$, and
\item\label{newcon} $d(\alpha\otimes1)+d(\alpha\otimes1)=d(1\otimes\varepsilon(\alpha))$.
\end{enumerate}
for all $a,b\in A$, $\alpha,\beta\in B$, and $\lambda,\mu\in\mathbbm{k}$.
\end{definition}

We note that of course $(\alpha\otimes a)(\beta\otimes b)=\alpha\beta\otimes ab$. Furthermore, notice that as consequence of the above, one can write
\begin{align*}
d(\alpha\beta\otimes ab)&=a\varepsilon(\alpha)d(\beta\otimes b)+b\varepsilon(\beta)d(\alpha\otimes a),\\
d(1\otimes ab)&=ad(1\otimes b)+bd(1\otimes a),\\
d(\alpha\beta\otimes1)&=\varepsilon(\alpha)d(\beta\otimes1)+\varepsilon(\beta)d(\alpha\otimes1), \text{and}\\
d(\alpha\otimes a)&=\varepsilon(\alpha)d(1\otimes a)+ad(\alpha\otimes1).
\end{align*}

\begin{remark}\label{zerorem}
Observe that $d(1\otimes1)=0$, and therefore $d(\lambda\otimes\mu)=0$ for all $\lambda,\mu\in\mathbbm{k}$. We also note that secondary K\"ahler differentials are in fact K\"ahler differentials of $B\otimes A$ as a commutative algebra (by way of \eqref{og1} and \eqref{og2}), with the addition of \eqref{newcon}.
\end{remark}

\begin{remark}
By taking $B=\mathbbm{k}$, notice that secondary K\"ahler differentials become the usual K\"ahler differentials $\Omega_{A|\mathbbm{k}}^1$. More specifically, \eqref{newcon} gets suppressed in Definition \ref{SKdiff} with the identification of $B=\mathbbm{k}$.
\end{remark}

\begin{proposition}\label{HH1AB}
For a commutative triple $\mathcal{T}=(A,B,\varepsilon)$, we have that
$$HH_1(A,B,\varepsilon)\cong\Omega_{\mathcal{T}|\mathbbm{k}}^1.$$
\end{proposition}
\begin{proof}
Since $A$ is commutative, the map $\partial_1:A^{\otimes2}\otimes B\longrightarrow A$ is trivial. Therefore we have that $HH_1(A,B,\varepsilon)$ is the quotient of $A^{\otimes2}\otimes B$ by the relation
\begin{equation}\label{rel}
\otimes\begin{pmatrix} ab\varepsilon(\alpha) & \beta\gamma\\ 1 & c\\ \end{pmatrix}-\otimes\begin{pmatrix} a & \alpha\beta\\ 1 & bc\varepsilon(\gamma)\\ \end{pmatrix}+\otimes\begin{pmatrix} ca\varepsilon(\beta) & \alpha\gamma\\ 1 & b\\ \end{pmatrix}=0.
\end{equation}
The map $HH_1(A,B,\varepsilon)\longrightarrow\Omega_{\mathcal{T}|\mathbbm{k}}^1$ sends the class of $\otimes\begin{pmatrix} a & \alpha\\ 1 & b\\ \end{pmatrix}$ to $ad(\alpha\otimes b)$. This map is well-defined because \eqref{rel} maps to
\begin{align*}
ab\varepsilon(\alpha)&d(\beta\gamma\otimes c)-ad(\alpha\beta\otimes bc\varepsilon(\gamma))+ca\varepsilon(\beta)d(\alpha\gamma\otimes b)\\
&=ab\varepsilon(\alpha)\varepsilon(\beta\gamma)d(1\otimes c)+ab\varepsilon(\alpha)cd(\beta\gamma\otimes1)-a\varepsilon(\alpha\beta)d(1\otimes bc\varepsilon(\gamma))\\
&~-abc\varepsilon(\gamma)d(\alpha\beta\otimes1)+ca\varepsilon(\beta)\varepsilon(\alpha\gamma)d(1\otimes b)+ca\varepsilon(\beta)bd(\alpha\gamma\otimes1)\\
&=ab\varepsilon(\alpha)\varepsilon(\beta\gamma)d(1\otimes c)+ab\varepsilon(\alpha)c\varepsilon(\beta)d(\gamma\otimes1)+ab\varepsilon(\alpha)c\varepsilon(\gamma)d(\beta\otimes1)\\
&~-a\varepsilon(\alpha\beta)bcd(1\otimes\varepsilon(\gamma))-a\varepsilon(\alpha\beta)b\varepsilon(\gamma)d(1\otimes c)-a\varepsilon(\alpha\beta)c\varepsilon(\gamma)d(1\otimes b)\\
&~-abc\varepsilon(\gamma)\varepsilon(\alpha)d(\beta\otimes1)-abc\varepsilon(\gamma)\varepsilon(\beta)d(\alpha\otimes1)+ca\varepsilon(\beta)\varepsilon(\alpha\gamma)d(1\otimes b)\\
&~+ca\varepsilon(\beta)b\varepsilon(\alpha)d(\gamma\otimes1)+ca\varepsilon(\beta)b\varepsilon(\gamma)d(\alpha\otimes1)\\
&=ab\varepsilon(\alpha)c\varepsilon(\beta)d(\gamma\otimes1)-a\varepsilon(\alpha\beta)bcd(1\otimes\varepsilon(\gamma))+ca\varepsilon(\beta)b\varepsilon(\alpha)d(\gamma\otimes1)\\
&=0.
\end{align*}

In the other direction, the map $\Omega_{\mathcal{T}|\mathbbm{k}}^1\longrightarrow HH_1(A,B,\varepsilon)$ sends $ad(\alpha\otimes b)$ to the class of $\otimes\begin{pmatrix} a & \alpha\\ 1 & b\\ \end{pmatrix}$, which is indeed a cycle because $ab\varepsilon(\alpha)-ba\varepsilon(\alpha)=0$, relying on $A$ being commutative. This is in fact a module homomorphism and we will now verify it respects the relations from Definition \ref{SKdiff}.

Observe that \eqref{og1} is immediate.

For \eqref{og2}, notice that $d(\alpha\beta\otimes bc)-b\varepsilon(\alpha)d(\beta\otimes c)-c\varepsilon(\beta)d(\alpha\otimes b)$ maps to
$$\otimes\begin{pmatrix} 1 & \alpha\beta\\ 1 & bc\\ \end{pmatrix}-\otimes\begin{pmatrix} b\varepsilon(\alpha) & \beta\\ 1 & c\\ \end{pmatrix}-\otimes\begin{pmatrix} c\varepsilon(\beta) & \alpha\\ 1 & b\\ \end{pmatrix},$$
which is $0$ by taking \eqref{rel} with $\gamma=1_B$ and $a=1_A$.

Finally for \eqref{newcon}, we note that $d(\gamma\otimes1)+d(\gamma\otimes1)-d(1\otimes\varepsilon(\gamma))$ maps to
$$\otimes\begin{pmatrix} 1 & \gamma\\ 1 & 1\\ \end{pmatrix}+\otimes\begin{pmatrix} 1 & \gamma\\ 1 & 1\\ \end{pmatrix}-\otimes\begin{pmatrix} 1 & 1\\ 1 & \varepsilon(\gamma)\\ \end{pmatrix},$$
which is $0$ from \eqref{rel} when we take $\alpha=\beta=1_B$ and $a=b=c=1_A$.

To complete the proof, we simply note that the two maps defined above are inverses of each other. That is,
$$\otimes\begin{pmatrix} a & \alpha\\ 1 & b\\ \end{pmatrix}\longmapsto ad(\alpha\otimes b)\longmapsto\otimes\begin{pmatrix} a & \alpha\\ 1 & b\\ \end{pmatrix},$$
and
$$ad(\alpha\otimes b)\longmapsto\otimes\begin{pmatrix} a & \alpha\\ 1 & b\\ \end{pmatrix}\longmapsto ad(\alpha\otimes b).$$
The isomorphism follows.
\end{proof}

\begin{corollary}\label{HC1AB}
For a commutative triple $\mathcal{T}=(A,B,\varepsilon)$, we have that
$$HC_1(A,B,\varepsilon)\cong\frac{\Omega_{\mathcal{T}|\mathbbm{k}}^1}{d(1\otimes A)}.$$
\end{corollary}
\begin{proof}
First note that $HC_0(A,B,\varepsilon)= HH_0(A,B,\varepsilon)$ by Example \ref{dimzer}, and then by Example \ref{dimzerH} and since $A$ is commutative, we have that $HH_0(A,B,\varepsilon)=\frac{A}{[A,A]}=A$. Hence $HC_0(A,B,\varepsilon)=A$. Therefore, employing Theorem \ref{SCLES} in low dimension, we get that
$$\ldots\longrightarrow A\xrightarrow{~B_*~}\Omega_{\mathcal{T}|\mathbbm{k}}^1\xrightarrow{~I_*~}HC_1(A,B,\varepsilon)\longrightarrow0.$$
Here $B:A\longrightarrow A^{\otimes2}\otimes B$ is defined by
$$B(a)=\otimes\begin{pmatrix} 1 & 1\\ 1 & a\\ \end{pmatrix}+\otimes\begin{pmatrix} a & 1\\ 1 & 1\\ \end{pmatrix}.$$
Then by exactness,
$$HC_1(A,B,\varepsilon)=\Ig(I_*)\cong\frac{\Omega_{\mathcal{T}|\mathbbm{k}}^1}{\Ker(I_*)}=\frac{\Omega_{\mathcal{T}|\mathbbm{k}}^1}{\Ig(B_*)}.$$
Finally we observe that the image of the map $B$ has the form $d(1\otimes a)+ad(1\otimes1)=d(1\otimes a)$, and this is exactly $d(1\otimes A)$, which was what we wanted.
\end{proof}

\begin{remark}
Again, taking $B=\mathbbm{k}$, we see that both Proposition \ref{HH1AB} and Corollary \ref{HC1AB} reduce to their well-known results in the usual case. That is, $HH_1(A)\cong\Omega_{A|\mathbbm{k}}^1$ and $HC_1(A)\cong\Omega_{A|\mathbbm{k}}^1/d(A)$, respectively.
\end{remark}

%%%%%%%%%%%%%%%%%%%%%%%%%%%%%%%%%%%%%%%%%%%%%%%%%%%%%%%%%%%%%%%%%%%%%%%%%%%%
\section{The Kernel of a Multiplication Map}\label{sec4}
%%%%%%%%%%%%%%%%%%%%%%%%%%%%%%%%%%%%%%%%%%%%%%%%%%%%%%%%%%%%%%%%%%%%%%%%%%%%

In this section, we establish the latter isomorphism corresponding to Theorem \ref{oldmain} for the secondary case.

For a commutative triple $(A,B,\varepsilon)$, set $J$ to be the kernel of the multiplication map from $A^{\otimes2}\otimes B$ to $A$. Explicitly, consider $m:A^{\otimes2}\otimes B\longrightarrow A$ given by $m\Big(\otimes\begin{pmatrix} a & \alpha\\ 1 & b\\ \end{pmatrix}\Big)=ab\varepsilon(\alpha)$, and set $J=\Ker(m)$. It's straightforward to see that $J$ is generated by elements of the form
$$
\otimes\begin{pmatrix} 1 & \alpha\\ 1 & a\\ \end{pmatrix}-\otimes\begin{pmatrix} a\varepsilon(\alpha) & 1\\ 1 & 1\\ \end{pmatrix}.
$$
Notice that when $B=\mathbbm{k}$ we get that $J$ reduces to $I$, as described in Section \ref{oldkernel}.

In order to capture the desired isomorphism, we must also introduce an additional set. Define $\hat{J}$ to be the set consisting of all finite sums of elements of the form
$$
\otimes\begin{pmatrix} 1 & \alpha\\ 1 & 1\\ \end{pmatrix}-\otimes\begin{pmatrix} \varepsilon(\alpha) & 1\\ 1 & 1\\ \end{pmatrix}+\otimes\begin{pmatrix} 1 & \alpha\\ 1 & 1\\ \end{pmatrix}-\otimes\begin{pmatrix} 1 & 1\\ 1 & 1\varepsilon(\alpha)\\ \end{pmatrix},
$$
where $\alpha\in B$. One can check that $\hat{J}$ is both an $A$-bimodule, as well as a submodule of $J$.

\begin{remark}
When $B=\mathbbm{k}$, observe that $\hat{J}=\{0\}\subseteq J^2=I^2$.
\end{remark}

One can also see that the $A$-bimodule $J/(J^2+\hat{J})$ is $A$-symmetric. That is, the two module structures agree. We are now ready for our sought-after isomorphism:

\begin{proposition}\label{JIso}
For a commutative triple $\mathcal{T}=(A,B,\varepsilon)$, we have that $$\Omega_{\mathcal{T}|\mathbbm{k}}^1\cong J/(J^2+\hat{J}).$$
\end{proposition}
\begin{proof}
We first define the map $\Omega_{\mathcal{T}|\mathbbm{k}}^1\longrightarrow J/(J^2+\hat{J})$ by sending $ad(\alpha\otimes b)$ to the class of $\otimes\begin{pmatrix} a & \alpha\\ 1 & b\\ \end{pmatrix}-\otimes\begin{pmatrix} a\varepsilon(\alpha)b & 1\\ 1 & 1\\ \end{pmatrix}$. Notice that this element is indeed in $J$. Next we aim to show that the map is well-defined and respects the relations found in Definition \ref{SKdiff}.

Of course, \eqref{og1} is clear.

For \eqref{og2}, notice that $d((\alpha\beta\otimes ab))-a\varepsilon(\alpha)d(\beta\otimes b)-b\varepsilon(\beta)d(\alpha\otimes a)$ maps to
\begin{align*}
\otimes&\begin{pmatrix} 1 & \alpha\beta\\ 1 & ab\\ \end{pmatrix}-\otimes\begin{pmatrix} ab\varepsilon(\alpha\beta) & 1\\ 1 & 1\\ \end{pmatrix}-\otimes\begin{pmatrix} a\varepsilon(\alpha) & \beta\\ 1 & b\\ \end{pmatrix}+\otimes\begin{pmatrix} ab\varepsilon(\alpha\beta) & 1\\ 1 & 1\\ \end{pmatrix}\\
&~-\otimes\begin{pmatrix} b\varepsilon(\beta) & \alpha\\ 1 & a\\ \end{pmatrix}+\otimes\begin{pmatrix} ab\varepsilon(\alpha\beta) & 1\\ 1 & 1\\ \end{pmatrix}\\
&=\otimes\begin{pmatrix} 1 & \alpha\beta\\ 1 & ab\\ \end{pmatrix}-\otimes\begin{pmatrix} a\varepsilon(\alpha) & \beta\\ 1 & b\\ \end{pmatrix}-\otimes\begin{pmatrix} b\varepsilon(\beta) & \alpha\\ 1 & a\\ \end{pmatrix}+\otimes\begin{pmatrix} ab\varepsilon(\alpha\beta) & 1\\ 1 & 1\\ \end{pmatrix}\in J^2\subseteq J^2 + \hat{J}.
\end{align*}

Finally, $d(\alpha\otimes1)+d(\alpha\otimes1)-d(1\otimes\varepsilon(\alpha))$ maps to
\begin{align*}
\otimes&\begin{pmatrix} 1 & \alpha\\ 1 & 1\\ \end{pmatrix}-\otimes\begin{pmatrix} \varepsilon(\alpha) & 1\\ 1 & 1\\ \end{pmatrix}+\otimes\begin{pmatrix} 1 & \alpha\\ 1 & 1\\ \end{pmatrix}-\otimes\begin{pmatrix} \varepsilon(\alpha) & 1\\ 1 & 1\\ \end{pmatrix}-\otimes\begin{pmatrix} 1 & 1\\ 1 & \varepsilon(\alpha)\\ \end{pmatrix}+\otimes\begin{pmatrix} \varepsilon(\alpha) & 1\\ 1 & 1\\ \end{pmatrix}\\
&=\otimes\begin{pmatrix} 1 & \alpha\\ 1 & 1\\ \end{pmatrix}-\otimes\begin{pmatrix} \varepsilon(\alpha) & 1\\ 1 & 1\\ \end{pmatrix}+\otimes\begin{pmatrix} 1 & \alpha\\ 1 & 1\\ \end{pmatrix}-\otimes\begin{pmatrix} 1 & 1\\ 1 & \varepsilon(\alpha)\\ \end{pmatrix}\in\hat{J}\subseteq J^2+\hat{J},
\end{align*}
and hence \eqref{newcon} is satisfied.

Next we define the map $J/(J^2+\hat{J})\longrightarrow\Omega_{\mathcal{T}|\mathbbm{k}}^1$ by sending the class of $\otimes\begin{pmatrix} 1 & \alpha\\ 1 & a\\ \end{pmatrix}-\otimes\begin{pmatrix} a\varepsilon(\alpha) & 1\\ 1 & 1\\ \end{pmatrix}$ to $1d(\alpha\otimes a)-a\varepsilon(\alpha)d(1\otimes1)=d(\alpha\otimes a)$. Notice this is well-defined because elements in $J^2$, which are of the form
\begin{align*}
&\left(\otimes\begin{pmatrix} 1 & \alpha\\ 1 & a\\ \end{pmatrix}-\otimes\begin{pmatrix} a\varepsilon(\alpha) & 1\\ 1 & 1\\ \end{pmatrix}\right)\left(\otimes\begin{pmatrix} 1 & \beta\\ 1 & b\\ \end{pmatrix}-\otimes\begin{pmatrix} b\varepsilon(\beta) & 1\\ 1 & 1\\ \end{pmatrix}\right)\\
&=\otimes\begin{pmatrix} 1 & \alpha\beta\\ 1 & ab\\ \end{pmatrix}-\otimes\begin{pmatrix} b\varepsilon(\beta) & \alpha\\ 1 & a\\ \end{pmatrix}-\otimes\begin{pmatrix} a\varepsilon(\alpha) & \beta\\ 1 & b\\ \end{pmatrix}+\otimes\begin{pmatrix} a\varepsilon(\alpha)b\varepsilon(\beta) & 1\\ 1 & 1\\\end{pmatrix}
\end{align*}
map to
$$
d(\alpha\beta\otimes ab)-b\varepsilon(\beta)d(\alpha\otimes a)-a\varepsilon(\alpha)d(\beta\otimes b)+a\varepsilon(\alpha)b\varepsilon(\beta)d(1\otimes1)=0
$$
by \eqref{og2} from Definition \ref{SKdiff} and Remark \ref{zerorem}. Likewise, elements in $\hat{J}$, which are of the form
$$
\otimes\begin{pmatrix} 1 & \alpha\\ 1 & 1\\ \end{pmatrix}-\otimes\begin{pmatrix} \varepsilon(\alpha) & 1\\ 1 & 1\\ \end{pmatrix}+\otimes\begin{pmatrix} 1 & \alpha\\ 1 & 1\\ \end{pmatrix}-\otimes\begin{pmatrix} 1 & 1\\ 1 & 1\varepsilon(\alpha)\\ \end{pmatrix}
$$
map to
$$
d(\alpha\otimes 1)-\varepsilon(\alpha)d(1\otimes 1)+d(\alpha\otimes 1)-d(1\otimes \varepsilon(\alpha))=0
$$
by \eqref{newcon} from Definition \ref{SKdiff} and Remark \ref{zerorem}.

To conclude, we notice that these maps are inverses of each other. That is,
$$
\otimes\begin{pmatrix} 1 & \alpha\\ 1 & a\\ \end{pmatrix}-\otimes\begin{pmatrix} a\varepsilon(\alpha) & 1\\ 1 & 1\\ \end{pmatrix}\longmapsto d(\alpha\otimes a)\longmapsto\otimes\begin{pmatrix} 1 & \alpha\\ 1 & a\\ \end{pmatrix}-\otimes\begin{pmatrix} a\varepsilon(\alpha) & 1\\ 1 & 1\\ \end{pmatrix},
$$
and
$$
ad(\alpha\otimes b)\longmapsto\otimes\begin{pmatrix} a & \alpha\\ 1 & b\\ \end{pmatrix}-\otimes\begin{pmatrix} a\varepsilon(\alpha)b & 1\\ 1 & 1\\ \end{pmatrix}\longmapsto ad(\alpha\otimes b).
$$
The isomorphism follows.
\end{proof}

We can now state our main theorem in full:

\begin{theorem}\label{newmain}
For a commutative triple $\mathcal{T}=(A,B,\varepsilon)$, we have that
$$HH_1(A,B,\varepsilon)\cong\Omega_{\mathcal{T}|\mathbbm{k}}^1\cong J/(J^2+\hat{J}).$$
\end{theorem}
\begin{proof}
Unite Proposition \ref{HH1AB} and Proposition \ref{JIso}.
\end{proof}

\begin{remark}
Taking $B=\mathbbm{k}$, we get that Theorem \ref{newmain} becomes Theorem \ref{oldmain}.
\end{remark}

%%%%%%%%%%%%%%%%%%%%%%%%%%%%%%%%%%%%%%%%%%%%%%%%%%%%%%%%%%%%%%%%%%%%%%%%%%%%
\section{Future Work}\label{sec5}
%%%%%%%%%%%%%%%%%%%%%%%%%%%%%%%%%%%%%%%%%%%%%%%%%%%%%%%%%%%%%%%%%%%%%%%%%%%%

There are rich properties involving $\Omega_{A|k}^1$ and $I/I^2$, such as the second fundamental exact sequence for $\Omega$. It is reasonable to ask if $\Omega_{\mathcal{T}|\mathbbm{k}}^1$ and $J/(J^2+\hat{J})$ also fit into some exact sequence. Furthermore, one of the main applications of Hochschild homology is that it leads to de Rham cohomology. It is appropriate that the secondary Hochschild homology could lead to a secondary de Rham cohomology, which could have interesting geometrical consequences and results. Ultimately, one wonders if we could derive an analogue of the Hochschild-Kostant-Rosenberg Theorem.

Finally, with the study of secondary Hochschild homology and secondary K\"ahler differentials, a natural follow-up is to investigate the secondary Hochschild cohomology and generalized derivations. We plan to delve into this more in the future.

%%%%%%%%%%%%%%%%%%%%%%%%%%%%%%%%%%%%%%%%%%%%%%%%%%%%%%%%%%%%%%%%%%%%%%%%%%%%

\end{document}